\documentclass{elsarticle}

\usepackage{amsmath,amssymb,amsthm,mathtools}
\usepackage{enumerate, enumitem}
\usepackage{blkarray, bigstrut}
\usepackage{booktabs}

\newcommand\topstrut[1][1.2ex]{\setlength\bigstrutjot{#1}{\bigstrut[t]}}
\newcommand\botstrut[1][0.9ex]{\setlength\bigstrutjot{#1}{\bigstrut[b]}}

\newtheorem{theorem}{Theorem}[section]
\newtheorem{proposition}[theorem]{Proposition}
\newtheorem{lemma}[theorem]{Lemma}
\newtheorem{corollary}[theorem]{Corollary}

\theoremstyle{definition}
\newtheorem{definition}[theorem]{Definition}
\newtheorem{example}[theorem]{Example}

\theoremstyle{remark}
\newtheorem{observation}[theorem]{Observation}

\newcommand{\R}{\mathbb{R}}
\newcommand{\skw}[1]{\text{\rm Skew$_{#1}$}}
\newcommand{\sym}[1]{\text{\rm Sym$_{#1}$}}

\newcommand{\M}{\mathcal{M}}
\newcommand{\Orth}{\mathcal{O}}
\newcommand{\Q}{\mathcal{Q}}
\newcommand{\T}{\mathcal{T}}
\newcommand{\norm}{\mathcal{N}}

\newcommand{\tr}{{\rm tr}}
\newcommand{\ts}{{\rm TS}}
\newcommand{\ns}{{\rm NS}}
\newcommand{\spn}[1]{{\rm Span}\{#1\}}
\newcommand{\ve}{{\rm Vec}}
\newcommand{\sgn}{{\rm sgn}}

\newcommand{\ba}{\mathbf{a}}
\newcommand{\bb}{\mathbf{b}}
\newcommand{\bc}{\mathbf{c}}
\newcommand{\bu}{\mathbf{u}}
\newcommand{\bv}{\mathbf{v}}
\newcommand{\bx}{\mathbf{x}}
\newcommand{\by}{\mathbf{y}}

\newcommand{\br}{\mathbf{r}}
\newcommand{\bk}{\mathbf{k}}
\newcommand{\bzero}{\mathbf{0}}

\newcommand\defi[1]{{\sl #1}}

\newcommand{\ip}[1]{\langle #1 \rangle}

\begin{document}
\begin{frontmatter}

\title{Sign Patterns of Orthogonal Matrices and the Strong Inner Product Property}

\author[1]{Bryan A. Curtis}
\ead{bcurtis6@uwyo.edu}

\author[1]{Bryan L. Shader}
\ead{bshader@uwyo.edu}

\address[1]{Department of Mathematics, University of Wyoming, Laramie, WY 82071, USA}

\begin{abstract}
A new condition, the strong inner product property, is introduced and used to construct sign patterns of row orthogonal matrices. Using this property, infinite families of sign patterns allowing row orthogonality are found. These provide insight into the underlying combinatorial structure of row orthogonal matrices.  Algorithmic techniques for verifying that a matrix has the strong inner product property are also presented. These techniques lead to a generalization of the strong inner product property and can be easily implemented using various software.
\end{abstract}

\end{frontmatter}

%
%
\section{Introduction}
Characterizing the sign patterns of orthogonal matrices has been of interest since the early 1960's. This endeavor was first proposed by M. Fiedler, in June 1963, at the Symposium on the Theory of Graphs and Its Applications \cite{Fiedler1964}. More recently there has been renewed interest in sign patterns of orthogonal matrices \cite{Beasley1993}, \cite{Beasley2006}, and in related qualitative and combinatorial problems \cite{BARRETT2015}, \cite{BARRETT2017}, \cite{CHEON2003}. There has been some success in characterizing the sign patterns of orthogonal matrices for small orders or with additional combinatorial constraints \cite{Gao2006}, \cite{JOHNSON1998}, \cite{WATERS1996}. As of the publication of this paper there is no characterization for orders $n\geq 6$.

For many years the only known necessary condition for a sign pattern to allow orthogonality was \defi{potential orthogonality}; that is the rows (respectively columns) are nonzero and the sign patterns of each pair of rows (respectively columns) have a realization that are orthogonal. The first example of a potentially orthogonal sign pattern not allowing orthogonality was given in 1996 \cite{WATERS1996}. Shortly after this observation, Johnson and Waters provided the first necessary condition stronger than potential orthogonality \cite{JOHNSON1998}. It is still not known whether this necessary condition is sufficient.

Developing sufficient conditions has also had some success in the literature. A common technique is to take a known orthogonal matrix and search for ``nearby'' orthogonal matrices. Givens rotations can be used to perturb certain zero entries of orthogonal matrices without affecting the sign of the nonzero entries \cite{CHEON2003}. The implicit function theorem has also been used in conjunction with special classes of orthogonal matrices \cite{WATERS1996}.  In this paper we introduce the strong inner product property, a tool that guarantees the existence of sign patterns of orthogonal matrices by perturbing the entries of ``nicely'' behaved orthogonal matrices. The strong inner product property surpasses previous methods in its ability to construct numerous examples of sign patterns of orthogonal matrices.

The next section provides the preliminary definitions and notation necessary to discuss sign patterns of orthogonal matrices. In Section 3 we introduce the strong inner product property and develop some basic results. Section 4 provides the motivation behind, and describes how to apply, the strong inner product property. In section 5 we consider some applications of the strong inner product property. We conclude with a generalization of the strong inner product property and a useful verification technique in Section 6.

%
%
\section{Preliminaries and Notation}

All matrices in this paper are real. Throughout, we restrict $m\leq n$ to be integers. The symbols $O$ and $I$ represent the zero and identity matrices of appropriate sizes, respectively. Let $\R^{m\times n}$ denote the set of all real $m\times n$ matrices, $\skw n$ the set of all $n\times n$ skew symmetric matrices, and $\sym n$ the set of all $n\times n$ symmetric matrices. A matrix $A$ (respectively vector $\bb$) is \defi{nowhere zero} if every entry in $A$ (respectively $\bb$) is nonzero. A matrix has \defi{full rank} if its rank is the largest possible. Define $E_{ij}\in\R^{m\times n}$ to be the matrix with a 1 in position $(i,j)$ and 0 elsewhere. If there is ever ambiguity in the dimensions of $E_{ij}$ we will specify. The set of $m \times n$ \defi{row orthogonal} matrices is
\[
\Orth(m,n) = \{Q\in\R^{m\times n}: QQ^T = I\};
\]
if $m = n$ we abbreviate this to $\Orth(n)$.

The \defi{support} of a matrix $A$ (respectively vector $\bb$) is the set of indices corresponding to the nonzero entries of $A$ (respectively $\bb$). For $A\in \R^{m\times n}$, $\alpha\subseteq\{1,2,\ldots,m\}$ and $\beta\subseteq\{1,2,\ldots,n\}$ the submatrix of $A$ with rows indexed by $\alpha$ and columns indexed by $\beta$ is denoted by $A[\alpha,\beta]$; in the case that $\beta = \{1,2,\ldots,n\}$ this is shortened to $A[\alpha,:]$ and similarly $A[:,\beta]$ for $\alpha = \{1,2,\ldots, m\}$. The \defi{Hadamard (entrywise) product} of the matrices $A$ and $B$ is denoted by $A\circ B$. For a matrix $A\in\R^{m\times n}$, $\ve(A)$ denotes the column vector of dimension $mn$ obtained by stacking together the columns of $A$. For example, if
\[
A =
\left[\begin{array}{ccc}
a_{11} & a_{12} & a_{13}\\
a_{21} & a_{22} & a_{23}\\
\end{array}\right],
\mbox{ then }
\ve(A) = 
\left[\begin{array}{c}
a_{11}\\ a_{21}\\ \hline a_{12}\\ a_{22}\\ \hline a_{13}\\ a_{23}
\end{array}\right].
\]
Notice $\ve(A)$ is indexed by the pairs $(i,j)$, $1\leq i\leq m$ and $1\leq j\leq n$, in reverse lexicographic order, and $\ve:\R^{m\times n}\to \R^{mn}$ is a bijective linear map.

The \defi{sign} of a real number $a$ is
\[
\sgn(a) = 
\begin{cases}
\hfill 1 & \text{if }a>0,\\
\hfill 0 & \text{if }a=0,\\
-1 & \text{if }a<0.
\end{cases}
\]
A \defi{sign pattern} is a $(0,1,-1)$-matrix and the sign pattern of a matrix $A = [a_{ij}]$, written $\sgn(A)$, is the sign pattern whose $(i,j)$-entry is $\sgn(a_{ij})$. Given an $m\times n$ sign pattern $S$, the \defi{qualitative set} of $S$ is
\[
\mathcal{Q}(S) = \{A\in \R^{m\times n}: \sgn(A) = S\}.
\]
The $m\times n$ sign pattern $S$ \defi{allows orthogonality} if there exists a (row) orthogonal matrix $Q\in \mathcal{Q}(S)$. The \defi{super pattern} of $S = [s_{ij}]$ in the direction of the $m\times n$ sign pattern $R = [r_{ij}]$ is the matrix $S_{\vec{R}}$ having $(i,j)$-entry equal to $s_{ij}$ if $s_{ij}\not=0$ and $r_{ij}$ otherwise. For example, if 
\[
S =
\left[\begin{array}{rrrr}
0 & 1 & 1 & 1\\
1 & 0 & 1 & -1\\
1 & -1 & 0 & 1\\
1 & 1 & -1 & 0
\end{array}\right]
\qquad \text{and} \qquad
R=
\left[\begin{array}{rrrr}
1 & 0 & 1 & -1\\
1 & -1 & -1 & 0\\
-1 & -1 & 0 & 1\\
1 & -1 & 0 & 1
\end{array}\right],
\]
then
\[
S_{\vec{R}} =
\left[\begin{array}{rrrr}
1 & 1 & 1 & 1\\
1 & -1 & 1 & -1\\
1 & -1 & 0 & 1\\
1 & 1 & -1 & 1
\end{array}\right].
\]

%
%
\section{Strong Inner Product Property}
We begin with the definition of the strong inner product property and some basic results.

\begin{definition}\label{DefSIPP}
The $m\times n$ matrix $M$ with $m\leq n$ has the \defi{strong inner product property} (SIPP) provided $M$ has full rank and $X = O$ is the only symmetric matrix satisfying $(XM)\circ M = O$.
\end{definition}

At first glance the strong inner product property may seem unnatural. However, as we will see in Section~\ref{develop}, the strong inner product property manifests when properly chosen manifolds intersect transversally.

In future sections we apply the SIPP to row orthogonal matrices. As row orthogonal matrices have full rank, the condition ``$M$ has full rank'' in Definition~\ref{DefSIPP} seems unnecessary. This condition is justified when taking a closer look at the motivation behind the SIPP. In particular, the techniques in Section~\ref{develop} can be applied to the family of matrices
\[
\{A\in\R^{m\times n}: AA^T = P\},
\]
where $P$ is a fixed positive definite matrix. When $P$ is positive definite the theory in Section~\ref{develop} leads to the above definition of the SIPP. Requiring full rank ensures $AA^T$ is positive definite. We focus on the SIPP for row orthogonal matrices but give its properties in the general setting.

The terminology SIPP follows that of the strong Arnol'd property which uses similar ideas to obtain results about the maximum nullity of a certain family of symmetric matrices associated with a graph \cite{Holst2009}. There have been other generalizations of the strong Arnol'd property in different settings \cite{BARRETT2015}. 

In the case that $m=n$ and $M$ is invertible the conditions required to have the SIPP can be simplified. 

\begin{theorem}\label{equivSIPP}
Suppose $M\in \R^{n\times n}$ is invertible. Then $M$ has the SIPP if and only if $Y = O$ is the only matrix such that $YM^{-1}$ is symmetric and $Y\circ M =O$.
\end{theorem}

\begin{proof}
Suppose that $M$ has the SIPP. Let $Y$ be an $n \times n$ matrix such that $YM^{-1}$ is symmetric and $Y\circ M=O$. Let $X=YM^{-1}$ so that $XM = Y$. Then $X$ is  symmetric and $(XM) \circ M=O$. Since $M$ has the SIPP, $X=O$ and consequently, $Y = O$.
 
Conversely, suppose that $Y = O$ is the only matrix such that $YM^{-1}$ is symmetric and $Y\circ M =O$. Let $X$ be a symmetric matrix such that $(XM) \circ M=O$. Let $Z=XM$ so that $X = ZM^{-1}$ is symmetric and $Z\circ M=(XM) \circ M=O$. By our assumptions $Z=O$. Hence $XM=Z=O$, and since $M$ is invertible, $X = O$. Therefore, $M$ has the SIPP.
\end{proof}

If we further restrict $M$ to be orthogonal then we have the following useful corollary.

\begin{corollary}\label{orthogSIPP}
Suppose $Q\in\Orth(n)$. Then $Q$ has the SIPP if and only if $Y=O$ is the only matrix $Y$ such that  the dot product between row $i$ of $Y$ and row $j$ of $Q$ equals the dot product between row $j$ of $Y$ and row $i$ of $Q$ for all $i$ and $j$, and $Y\circ Q = O$.
\end{corollary}
\begin{proof}
This equivalence follows from Theorem~\ref{equivSIPP} and the observation that the $(i,j)$-entry of $YQ^T$ is the dot product between the $i$-th row of $Y$ and the $j$-th row of $Q$. 
\end{proof}

When studying (row) orthogonal matrices there are two convenient types of equivalence. A \defi{signed permutation matrix} is a square sign pattern with exactly one nonzero entry in each row and column. Matrices $A,B\in\R^{m\times n}$ are \textit{sign equivalent} if $A = P_1BP_2$, where $P_1$ and $P_2$ are signed permutation matrices. If $m = n$, the matrices $A$ and $B$ are \textit{equivalent} if $A$ is sign equivalent to $B$ or $B^T$.  Both forms of equivalence preserve (row) orthogonality and the combinatorial structure of the corresponding sign patterns. Not surprisingly, sign equivalence preserves having the SIPP. However, the same cannot always be said about equivalence (see Example~\ref{equivCounterEx}).

\begin{lemma}\label{sign equiv}
Let $A, B\in \R^{m\times n}$ be sign equivalent. Then $A$ has the SIPP if and only if $B$ has the SIPP.
\end{lemma}
\begin{proof}
It suffices to assume that $A$ has the SIPP and show that $B$ has the SIPP. Since $A$ and $B$ are sign equivalent, $B = P_1 A P_2$, where $P_1$ and $P_2$ are signed permutation matrices. Hence $B$ has full rank. Let $X\in \sym m$ and $Y=P_1^TXP_1$. Then $Y$ is symmetric and $X = P_1 Y P_1^T$. Suppose that $(XB)\circ B = O$. Then 
\[
(P_1 YA P_2)\circ (P_1 A P_2) = O
\]
and so $(YA)\circ A = O$. Since $A$ has the SIPP, $Y = O$. Consequently, 
\[
X = P_1 Y P_1^T = O.
\]
Therefore, $B$ has the SIPP.
\end{proof}

We now show that if $Q_1, Q_2\in \Orth(n)$ are equivalent, then $Q_1$ has the SIPP if and only if $Q_2$ has the SIPP. By Lemma~\ref{sign equiv} it suffices to prove the case $Q_1 = Q_2^T$.

\begin{proposition}\label{transpose}
Let $Q \in \Orth(n)$. Then $Q$ has the SIPP if and only if $Q^T$ has the SIPP.
\end{proposition}
\begin{proof}
It suffices to assume that $Q$ has the SIPP and show that $Q^T$ has the SIPP. Let $X\in \R^{n\times n}$. Suppose that $XQ^{-T}$ is symmetric and $X\circ Q^T = O$. By Theorem~\ref{equivSIPP} it remains to show that $X = O$. Note that $XQ^{-T} = XQ$. Since $XQ = Q^TX^T$,
\[
X^TQ^{-1} = QQ^TX^TQ^T = QXQQ^T = QX = (X^TQ^{-1})^T.
\]
Thus, $X^TQ^{-1}$ is symmetric. Further, $X^T\circ Q = O$ since $X\circ Q^T = O$. Having assumed $Q$ has the SIPP, $X^T = X = O$ and so $Q^T$ has the SIPP.
\end{proof}

If $M\in\R^{m\times n}$ is nowhere zero and has full rank, then $M$ has the SIPP (if $(XM)\circ M = O$, then $XM = O$ implying that $X = O$). On the other hand, Corollary~\ref{zeros} suggests that the SIPP becomes exceedingly rare amongst sparse matrices. Lemma~\ref{basics} demonstrates that matrices must avoid certain structural barriers in order to have the SIPP.

\begin{lemma}\label{basics}
Let $M\in \R^{m\times n}$. If $M$ has two rows with disjoint support, then $M$ does not have the SIPP.
\end{lemma}
\begin{proof}
Assume that $M\in\R^{m\times n}$ has two rows with disjoint support. Up to permutation of rows and columns $M$ has the form
\[
\renewcommand*{\arraystretch}{1.4}
M =
\left[\begin{array}{c|c}
\bu^T & \bzero^T\\
\hline
\bzero^T & \bv^T \\
\hline
C & D
\end{array}\right],
\]
where $\bu$ and $\bv$ are nonzero. Observe that the $m\times m$ symmetric matrix
\[
X =
\left[\begin{array}{cc}
0 & 1\\
1 & 0
\end{array}\right]
\oplus
O.
\]
satisfies $(XM)\circ M = O$. However, $XM\not=O$ implying that $X \not= O$. Therefore, the matrix $M$ does not have the SIPP.
\end{proof}

There is no analog of Lemma~\ref{basics} for the columns of a matrix $M\in\R^{m\times n}$, i.e.~$M$ can have columns with disjoint support and the SIPP (see Proposition~\ref{mnSIPP}). Having established Lemma~\ref{basics} we can now show that Proposition~\ref{transpose} does not hold for arbitrary invertible square matrices.

\begin{example}\label{equivCounterEx}
Consider
\[
A =
\left[\begin{array}{rrr}
2 & 0 & 0\\
0 & 1 & 1\\
-2 & -1 & 1
\end{array}\right].
\]
Then
\[
A^{-1} = \frac12
\left[\begin{array}{rrr}
1 & 0 & 0\\
-1 & 1 & -1\\
1 & 1 & 1
\end{array}\right].
\]
By Lemma~\ref{basics} $A$ does not have the SIPP.

On the other hand, $A^T$ does have the SIPP. To see this, let $X\in\R^{3\times 3}$ satisfy $X\circ A^T = O$. Then $X$ must have the form
\[
X =
\left[\begin{array}{rrr}
0 & x_1 & 0\\
x_2 & 0 & 0\\
x_3 & 0 & 0
\end{array}\right],
\text{ and }
XA^{-T} =
\frac12
\left[\begin{array}{rrr}
0 & x_1 & x_1\\
x_2 & -x_2 & x_2\\
x_3 & -x_3 & x_3
\end{array}\right].
\]
Assuming $XA^{-T}$ is symmetric implies $x_1 = x_2 = x_3 = 0$. 
\end{example}

The proof of Proposition~\ref{transpose} relies on the fact that $Q^T$ is the inverse of $Q$ when $Q\in\Orth(n)$. It may therefore be tempting to try and prove that if $A$ has the SIPP then so must $A^{-1}$. As the next example illustrates, this is not always the case. 

\begin{example}
Consider
\[
A =
\left[\begin{array}{rrr}
1 & 1 & 0\\
0 & 1 & 1\\
1 & 0 & 1
\end{array}\right].
\]
Then
\[
A^{-1} = \frac12
\left[\begin{array}{rrr}
1 & -1 & 1\\
1 & 1 & -1\\
-1 & 1 & 1
\end{array}\right].
\]
Since $A^{-1}$ is nowhere zero, it has the SIPP. To see that $A$ does not have the SIPP, choose
\[
X =
\left[\begin{array}{rrr}
0 & 0 & 1\\
1 & 0 & 0\\
0 & 1 & 0
\end{array}\right],
\]
and note that $X\circ A = O$ and $XA^{-1}$ is symmetric.
\end{example}

Occasionally it is possible to verify that a matrix has the SIPP by checking if a certain submatrix has the SIPP.

\begin{proposition}\label{removeRow}
Assume $2\leq m\leq n-1$. Let $\hat{B}\in\R^{m\times n}$ and $B = \begin{bmatrix}\hat{B}\\ \cmidrule(lr){1-1} \bb^T \end{bmatrix}$ for some $\bb\in\R^n$.
\begin{enumerate}
\item[\rm (i)]
If $B$ has the SIPP, then $\hat{B}$ has the SIPP.

\item[\rm (ii)]
If $\hat{B}$ has the SIPP, the rows of $B$ are linearly independent and $\bb$  is nowhere zero, then $B$ has the SIPP.
\end{enumerate}
\end{proposition}
\begin{proof}
We begin by proving (i). Assume $B$ has the SIPP. Let $\hat{X}\in \sym m$ satisfy $(\hat{X}\hat{B})\circ\hat{B} = O$. We must show that $\hat{X} = O$. Consider the $(m+1)\times(m+1)$ symmetric matrix
\[
\renewcommand*{\arraystretch}{1.4}
X =
\left[\begin{array}{c|c}
\hat{X} & \bzero\\
\hline
\bzero^T & 0
\end{array}\right].
\]
Then
\[
(XB)\circ B = 
\begin{bmatrix}
(\hat{X}\hat{B})\circ\hat{B}\\
\cmidrule(lr){1-1}
\bzero^T
\end{bmatrix}
 = O.
\]
Since $B$ has the SIPP we know that $X = O$. Thus, $\hat{X} = O$ and so $\hat{B}$ has the SIPP.

We now prove (ii). Assume $\hat{B}$ has the SIPP, the rows of $B$ are linearly independent and $\bb$  is nowhere zero. Let $X\in \sym{m+1}$ satisfy $(XB)\circ B = O$. We must show that $XB = O$. Observe that $X$ has the form 
\[
\renewcommand*{\arraystretch}{1.4}
X = 
\left[\begin{array}{c|c}
\hat{X} & \bx\\
\hline
\bx^T & x_1 
\end{array}\right]
\]
for some $\hat{X}\in\sym m$, $\bx\in \R^m$ and $x_1\in\R$. Then
\[
O = (XB) \circ B = \begin{bmatrix} (\hat{X}\hat{B} + \bx\bb^T) \circ \hat{B}\\\cmidrule(lr){1-1} (\bx^T\hat{B} + x_1\bb^T)\circ \bb^T\end{bmatrix}
\]
and so
\begin{align}
O &=( \hat{X}\hat{B} + \bx\bb^T) \circ \hat{B} \label{e1},\\
\bzero^T &= (\bx^T\hat{B} + x_1\bb^T)\circ \bb^T \label{e2}.
\end{align}
 Since $\bb$ is nowhere zero, (\ref{e2}) becomes
\[
\bx^T\hat{B} + x_1\bb^T = \bzero^T.
\]
Having assumed the rows of $B$ are linearly independent, we conclude that $\bx =\bzero$ and $x_1 = 0$. Since $\bx = \bzero$, (\ref{e1}) reduces to $(\hat{X}\hat{B})\circ \hat{B} = O$ and since $\hat{B}$ has the SIPP, $\hat{X} = O$. Hence $X=O$ proving that $B$ has the SIPP.
\end{proof}

\begin{proposition}\label{mnSIPP}
Let $A\in\R^{m\times n}$ and $p>n$. Then $A$ has the SIPP if and only if the $m\times p$ matrix $B = \left[\begin{array}{@{}c|c@{}}A & O\end{array}\right]$ has the SIPP.
\end{proposition}
\begin{proof}
Clearly $A$ has full rank if and only if $B$ has full rank. 

Begin by assuming $A$ has the SIPP. Let $X \in \sym m$ and suppose $(XB) \circ B = O$. Then $(XA)\circ A = O$ and since $A$ has the SIPP, we have $X = O$. Thus, $B$ has the SIPP.

Now assume $B$ has the SIPP. Let $X \in \sym m$ and suppose $(XA) \circ A = O$. Then $(XB)\circ B = O$ and since $B$ has the SIPP $X = O$. Thus, $A$ has the SIPP.
\end{proof}

%
%
\section{Development and Motivation Behind the SIPP}\label{develop}
The primary goal in this section is to motivate and rigorously develop the SIPP. We will also show how to use families of (preferably sparse) matrices with the SIPP to obtain larger families of sign patterns that allow orthogonality. In order to accomplish this, it is necessary to understand some facts about smooth manifolds. We refer the reader to \cite{Lee2003} for more details.  

Let $\M$ be a smooth manifold in $\R^d$ and let $\bx\in \M$. Define $\mathcal{P}_\M$ to be the set of smooth paths $\gamma:(-1,1)\to\M$, and let $\dot{\gamma}$ be the derivative of $\gamma$ 
with respect to $t \in (-1,1)$. The \defi{tangent space} of $\M$ at $\bx$ is
\[
\T_{\M\cdot \bx} = \{\dot{\gamma}(0) : \gamma\in\mathcal{P}_{\M} \text{ and }\gamma(0) = \bx\}
\]
and the \defi{normal space} to $\M$ at $\bx$, denoted $\norm_{\M\cdot \bx}$, is the orthogonal complement of $\T_{\M\cdot \bx}$. 
Note that as vector spaces 
\[
\dim(\T_{\M\cdot \bx}) = \dim(\M), \text{ and } \dim(\norm_{\M\cdot \bx}) = d - \dim(\T_{\M\cdot \bx}).
\]

Let $A,B\in \R^{m\times n}$ and $S$ be an $m\times n$ sign pattern. Both $\Orth(m,n)$ and $\Q(S)$ are smooth manifolds \cite{Lee2003} and can both be thought of as submanifolds of $\R^{mn}$ with inner product 
\[
\ip{A,B} = \tr(AB^T) = \ve(A)^T\ve(B).
\]
This identification follows from $\ve:\R^{m\times n}\to\R^{mn}$. Notice that there exists a matrix $Q\in\Orth(m,n)$ with sign pattern $S$ if and only if the intersection of $\Orth(m,n)$ and $\Q(S)$ is nonempty.

The smooth manifolds $\M_1$ and $\M_2$, both in $\R^d$, intersect \defi{transversally} at $\bx$ if 
\[
\bx\in \M_1\cap \M_2, \text{ and } \T_{\M_1\cdot \bx} + \T_{\M_2\cdot \bx} = \R^d
\]
or equivalently 
\[
\bx\in \M_1\cap \M_2, \text{ and } \norm_{\M_1\cdot \bx} \cap \norm_{\M_2\cdot \bx} = \{\bzero\}.
\]
A \defi{smooth family of manifolds} $\M(t)$ in $\R^d$ is defined by a continuous function $\varphi:U\times(-1,1)\to \R^d$, where $U$ is an open set in $\R^d$ and for each $t\in(-1,1)$ the function $\varphi(\cdot,t)$ is a diffeomorphism between $U$ and the manifold $M(t)$. Theorem~\ref{IFT} below is a specialization of Lemma 2.2 in \cite{Holst2009} and is stated with proof in \cite{BARRETT2015}.

\begin{theorem}[Holst \textit{et al.}~\cite{Holst2009}]\label{IFT}
Let $\M_1(s)$ and $\M_2(t)$ be smooth families of manifolds in $\R^d$, and assume that $\M_1(0)$ and $\M_2(0)$ intersect transversally at $y_0$. Then there is a neighborhood $W\subseteq \R^2$ of the origin and a continuous function $f:W\to \R^d$ such that $f(0,0) = y_0$ and for each $\epsilon = (\epsilon_1, \epsilon_2)\in W$, $\M_1(\epsilon_1)$ and $\M_2(\epsilon_2)$ intersect transversally at $f(\epsilon)$.
\end{theorem}

It is useful to think of Theorem~\ref{IFT} as saying if the manifolds $\M_1$ and $\M_2$ intersect transversally, then small perturbations of $\M_1$ and $\M_2$ still intersect (transversally) in a continuous way. Our aim is to apply these ideas to $\Orth(m,n)$ and $\Q(S)$. This requires the appropriate tangent and normal spaces.

The Stiefel manifold $\text{St}(n,m) = \{X \in \R^{n\times m}: X^T X = I\}$. Observe that $X\in\text{St}(n,m)$ if and only if $X^T\in\Orth(m,n)$. Since this identification preserves dimension we may use the calculation of $\dim(\text{St}(n,m))$ in \cite{Absil2008} to obtain
\begin{equation}\label{dim}
\dim(\Orth(m,n)) = nm - \frac12m(m+1).
\end{equation}
Note that for every $Q\in\Orth(m,n)$ there exists a matrix $P\in\Orth(n)$ such that $QP^T = [I\ O]$; $P$ is not unique unless $m=n$. With these observations we now compute the tangent space of $\Orth(m,n)$.

\begin{lemma}\label{TOmn}
Let $Q\in\Orth(m,n)$. Then
\[
\T_{\Orth(m,n)\cdot Q} = \{X\in\R^{m\times n}: XQ^T \in\skw{m}\}.
\]
Further, there exists some $P\in\Orth(n)$ such that
\[
\T_{\Orth(m,n)\cdot Q} = \left\{\left[\begin{array}{@{}c|c@{}}K & L\end{array}\right]P: K\in\skw{m} \text{ and } L\in\R^{m\times(n-m)}\right\}.
\]
\end{lemma}
\begin{proof}
Let $\gamma(t)$ be a differentiable path in $\Orth(m,n)$ such that $\gamma(0) = Q$. Then $\gamma(t)\gamma(t)^T = I$ and by taking derivatives of both sides
\[
\dot{\gamma}(0)Q^T + Q\dot{\gamma}(0)^T = O.
\]
It follows that $\T_{\Orth(m,n)\cdot Q}$ is contained in the vector space
\[
V =\{X\in\R^{m\times n}: XQ^T \in \skw m\}.
\]
Let $P\in\Orth(n)$ satisfy $QP^T= \left[\begin{array}{@{}c|c@{}}I & O\end{array}\right]$. For each $X\in \R^{m\times n}$ let $Y = XP^T$ so that $YP = X$. Using the substitution $Y = XP^T$
\begin{align*}
V &=\{X\in\R^{m\times n}: XP^TPQ^T \in \skw m\}\\
&=\{YP: Y \left[\begin{array}{@{}c|c@{}}I & O\end{array}\right]^T \in \skw m\}.
\end{align*}
Notice that $Y [I\ O]^T \in \skw m$ implies $Y$ has the form $\left[\begin{array}{@{}c|c@{}}K & L\end{array}\right]$ for some $K\in\skw m$ and $L\in\R^{m\times(n-m)}$. Thus,
\[
V=\left\{\left[\begin{array}{@{}c|c@{}}K & L\end{array}\right] P: K\in\skw{m} \text{ and } L\in\R^{m\times(n-m)}\right\}.
\]
It follows from the invertibility of $P$ that $V$ has the same dimension as
\[
\left\{\left[\begin{array}{@{}c|c@{}}K & L\end{array}\right]: K\in \skw m \text{ and }L\in \R^{m\times(n-m)}\right\}.
\]
Thus,
\[
\dim(V) = \binom{m}{2} + m(n-m) = mn - \binom{m+1}{2} = \dim(\Orth(m,n)),
\]
where the last equality comes from (\ref{dim}). Therefore, $\T_{\Orth(m,n)\cdot Q} = V$. 
\end{proof}

The second representation of $\T_{\Orth(m,n)\cdot Q}$ in Lemma~\ref{TOmn} will prove useful in Section~\ref{SecVerif} when we introduce the verification matrices. The first representation facilitates the  proof of the following lemma. We next determine the normal space $\norm_{\Orth(m,n)\cdot Q}$.

\begin{lemma}\label{NOmn}
Let $Q\in\Orth(m,n)$. Then
\[
\norm_{\Orth(m,n)\cdot Q} =\{ZQ : Z\in\sym m\}.
\]
\end{lemma}
\begin{proof}
Let
\[
W = \{ZQ: Z \in \sym m\}.
\]
Let $XQ^T\in\skw{m}$ and $Z\in\sym m$. Then
\[
\ip{ZQ,X} = \tr((ZQ)^TX) = \tr(Q^T Z X) = \tr(ZXQ^T) = 0,
\]
with the last equality coming from $Z$ being symmetric and $XQ^T$ being skew-symmetric.
Hence, by Lemma~\ref{TOmn}, $W$ is contained in $\norm_{\Orth(m,n)\cdot Q}$. Since the rows of $Q$ are linearly independent, the dimension of $W$ is equal to the dimension of $\sym m$. Thus,
\[
\dim(W) = \binom{m+1}{2} = mn - \dim(\T_{\Orth(m,n)\cdot Q}),
\]
where the second equality follows from (1). We conclude that $\norm_{\Orth(m,n)\cdot Q} = W$.
\end{proof}

We now compute the tangent and normal spaces of the manifold  $\Q(S)$ of matrices with a given sign pattern $S$.

\begin{lemma}\label{Qs}
Let $A\in\R^{m\times n}$ have sign pattern $S$. Then
\[
\T_{\Q(S)\cdot A} = \spn{E_{ij}: a_{ij}\not=0}
\]
and
\[
\norm_{\Q(S)\cdot A} = \{X : X\circ A = O\}.
\]
\end{lemma}
\begin{proof}
For each $i$ and $j$, where $a_{ij}\not=0$, define $\gamma$ to be the path on $\Q(S)$ given by $\gamma(t) = A + tE_{ij}$. Then $\gamma(0) = A$ and $\dot{\gamma}(0) = E_{ij}$. Thus, 
\[ 
\T_{\Q(S)\cdot A} \supseteq \spn{E_{ij}: a_{ij}\not=0}.
\]
Now let $\gamma \in \mathcal{P}_{\Q(S)}$ such that $\gamma(0) = A$. Then $\gamma(t)\in \spn{E_{ij}: a_{ij}\not=0}$ for each $t\in(-1,1)$ and so $\dot{\gamma}(0)\in \spn{E_{ij}: a_{ij}\not=0}$. Therefore, $\T_{\Q(S)\cdot A} = \spn{E_{ij}: a_{ij}\not=0}$, as desired.

For the normal space we have
\begin{align*}
\norm_{\Q(S)\cdot A} &=\{X : \tr(X^TY)=0 \text{ for all } Y\in \T_{\Q(S)\cdot A}\}\\
&= \{X : X\circ A = O\}.\qedhere
\end{align*}
\end{proof}

We note that if $P$ is an $m\times m$, positive definite matrix, then 
\[
\{ M\in \mathbb{R}^{m\times n}: MM^T=P\}
\]
is a smooth manifold \cite{Lee2003}. Hence one can talk about its tangent and normal space, and how it intersects the manifold of a given sign pattern transversally at a given matrix. Thus, similar results to Lemmas~\ref{TOmn} - \ref{NOmn} and Theorem~\ref{SIPPTran} hold. We do not need that level of generality in this study.

We now show how the transverse intersection of $\Orth(m,n)$ and $\Q(S)$ is related to the SIPP and how a (row) orthogonal matrix having the SIPP gives rise to many sign patterns that allow orthogonality. This will be used heavily in Section~\ref{SecFamilies}.

\begin{theorem}\label{SIPPTran}
Let $Q\in\Orth(m,n)$ have sign pattern $S= [s_{ij}]$. The manifolds $\Q(S)$ and $\Orth(m,n)$ intersect transversally at $Q$ if and only if $Q$ has the SIPP. Further, if $Q$ has the SIPP, then every super pattern of $S$ allows orthogonality.
\end{theorem}
\begin{proof}
By Lemmas \ref{NOmn} and \ref{Qs},
\[
\norm_{\Orth(m,n)\cdot Q} \cap \norm_{\Q(S)\cdot Q} = \{XQ : X\in\sym m \text{ and } XQ\circ Q = O\}.
\]
Thus, $\norm_{\Orth(m,n)\cdot Q} \cap \norm_{\Q(S)\cdot Q} = \{O\}$ if and only if $Q$ has the SIPP.

Suppose that $Q$ has the SIPP and let $R = [r_{ij}]$ be any sign pattern. Define the smooth family of manifolds $\M_R(t)$ by
\begin{align*}
\M_R(t) = \{A = [a_{ij}]\in\R^{m\times n} :&\ \sgn(a_{ij}) = s_{ij} \text{ if } s_{ij}\not=0 \\
&\text{ and } a_{ij} = r_{ij}t \text{ if } s_{ij} = 0\}
\end{align*}
for $t\in(-1,1)$. Then the manifolds $\Orth(m,n)$ and $\M_R(0) = \Q(S)$ intersect transversally at $Q$. Thus, for $\epsilon>0$ sufficiently small, Theorem~\ref{IFT} guarantees a continuous function $f$ such that $\M_R(\epsilon)$ and $\Orth(m,n)$ intersect transversally at $f(\epsilon)$ and $f(0) = Q$. Since $f(\epsilon)\in \M_R(\epsilon)$ we know that for $\epsilon$ sufficiently small $f(\epsilon)$ has sign pattern $S_{\vec{R}}$. Further, $f(\epsilon)\in\Orth(m,n)$ implies $S_{\vec{R}}$ allows orthogonality.
\end{proof}

Theorem~\ref{SIPPTran} is most effective at producing new sign patterns allowing orthogonality when $Q$ is sparse. The following corollary gives a bound on the number of zero entries a matrix with the SIPP can have. The bound in Corollary~\ref{zeros} is sharp as can be seen in Example~\ref{biplane}.

\begin{corollary}\label{zeros}
Let $Q\in\Orth(m,n)$ have the SIPP. Then the number of zero entries in $Q$ is bounded above by $ nm - \frac12m(m+1)$.
\end{corollary}
\begin{proof}
Let $S$ be the sign pattern of $Q$ and $N_0(Q)$ denote the number of zero entries in $Q$. By (\ref{dim}), $\dim(\norm_{\Orth(m,n)\cdot Q}) = \frac12m(m+1)$. Further, by Theorem~\ref{SIPPTran}, the matrix $Q$ has the SIPP if and only if $\norm_{\Orth(m,n)\cdot Q} \cap \norm_{\Q(S)\cdot Q} = \{O\}$. We conclude that
\begin{align*}
N_0(Q) &= \dim(\norm_{\Q(S)\cdot Q})\\
&\leq nm - \dim(\norm_{\Orth(m,n)\cdot Q}) \\
&= nm - \frac12m(m+1). \qedhere
\end{align*}
\end{proof}

%
%
\section{Families of Sign Patterns and Consequences of the SIPP}\label{SecFamilies}

We now give some consequences of matrices having the SIPP and several examples of families of sign patterns which allow orthogonality. The first example follows immediately from Proposition~\ref{mnSIPP} and Theorem~\ref{SIPPTran}.

\begin{corollary}\label{mnSIPPsuper}
Let $Q\in \Orth(m,n)$ have sign pattern $S$ and let $R$ be an $m\times p$ sign pattern with $p > n$. If $Q$ has the SIPP then the $m\times p$ sign pattern $\left[\begin{array}{@{}c|c@{}}S & O\end{array}\right]_{\vec{R}}$ allows orthogonality.
\end{corollary}

Considering Corollary~\ref{mnSIPPsuper}, it is reasonable to ask if an $m\times n$ sign pattern allows orthogonality, then must it contain an $m\times m$ subpattern that allows orthogonality. As was pointed out in the final remark of \cite{JOHNSON1998}, the sign pattern
\[
\left[\begin{array}{rrrrr}
-1&1&-1&1&1\\
1&-1&1&-1&1\\
1&1&1&1&-1\\
1&1&1&1&1
\end{array}\right]
\]
allows orthogonality, but does not have a $4\times 4$ submatrix that allows orthogonality.

Our next example concerns orthogonal Hessenberg matrices. The $n\times n$ Hessenberg matrix
\[
H = 
\left[\begin{array}{ccccc}
1 & -1 & 0 & \cdots & 0\\
\vdots & \ddots & -2 & \ddots & \vdots\\
\vdots & & \ddots & \ddots & 0\\
\vdots & & & \ddots & -(n - 1)\\
1 & \cdots & \cdots & \cdots & 1
\end{array}\right].
\]
is row orthogonal. Hence for each $n\geq 2$ there is a least one orthogonal Hessenberg matrix. Remarkably, the proof that orthogonal Hessenberg matrices have the SIPP does not depend on the signs of its entries. This is not always the case.

\begin{corollary}\label{Hess}
Let $n\geq 2$ and $Q\in\Orth(n)$ have zero-nonzero pattern 
\[
R = 
\left[\begin{array}{rrrrr}
1 & 1 & 0 & \cdots & 0\\
\vdots & \ddots & \ddots & \ddots & \vdots\\
\vdots & & \ddots & \ddots & 0\\
\vdots & & & \ddots & 1\\
1 & \cdots & \cdots & \cdots & 1
\end{array}\right].
\]
Then every super pattern of $\sgn(Q)$ allows orthogonality.
\end{corollary}
\begin{proof}
Suppose that $X=[x_{ij}]$ is a matrix such that $X\circ Q = O$ and the dot product between the $i$-th row of $X$ and $j$-th row of $Q$ equals the dot product between the $j$-th row of $X$ and the $i$-th row of $Q$ for all $i$ and $j$. By Corollary~\ref{orthogSIPP}, it suffices to show that $X=O$.

Since $X\circ Q=O$,  $x_{ij} = 0$ whenever $j \leq i + 1$. Consequently, row $i$ of $X$ and row $j$ of $Q$ are orthogonal whenever $j\leq i$. By our assumptions, row $j$ of $X$ and row $i$ of $Q$ are orthogonal whenever $j\leq i$. Hence, each row of $X$ is orthogonal to every row of $Q$. Since the rows of $Q$ form an orthonormal basis of $\R^n$, we conclude that $X=O$. Therefore, $Q$ has the SIPP and by Theorem~\ref{SIPPTran} every super pattern of $S$ allows orthogonality.
\end{proof}

In \cite{CHEON2003} Givens rotations are used to find sign patterns that allow orthogonality. As we will see, the techniques used there are not sufficient to prove Corollary~\ref{Hess}. 

A \defi{Givens rotation} is an orthogonal matrix that fixes all but two of the coordinate axes and rotates the plane spanned by these two coordinate axes. For $i < j$ let $G_{ij}(\theta)$ denote the Givens rotation which rotates the $(i,j)$-plane counterclockwise by $\theta$ radians. 

Let $A\in\R^{m\times n}$ have columns $\bc_1,\ldots,\bc_n$ and rows $\br_1,\ldots,\br_m$. Postmultiplying $A$ by $G_{ij}(\theta)$ replaces column $\bc_i$ by $\cos(\theta)\bc_i + \sin(\theta)\bc_j$, replaces column $\bc_j$ by $\cos(\theta)\bc_j -\sin(\theta)\bc_i$ and does not affect the remaining entries of $A$. Similarly, premultiplying $A$ by $G_{ij}(\theta)$ replaces row $\br_i$ by $\cos(\theta)\br_i - \sin(\theta)\br_j$, replaces row $\br_j$ by $\cos(\theta)\br_j + \sin(\theta)\br_i$ and does not affect the remaining entries of $A$.

Let $Q\in\Orth(m,n)$ have sign pattern $S$. Let $R_{ij}$ denote the sign pattern obtained from $S$ by replacing row $i$ and row $j$ by the negative of row $j$ and row $i$ respectively. Let $C_{ij}$ denote the sign pattern obtained from $S$ by replacing column $i$ and column $j$ by column $j$ and the negative of column $i$ respectively. Then there exists a $\theta>0$ such that 
\begin{align*}
\sgn\big(G_{ij}(\theta) Q\big) &= S_{\vec{R}_{ij}}, & \sgn\big(G_{ij}(-\theta) Q\big) &= S_{-\vec{R}_{ij}},\\
\sgn\big(QG_{ij}(\theta)\big) &= S_{\vec{C}_{ij}}, & \sgn\big(QG_{ij}(-\theta)\big) &= S_{-\vec{C}_{ij}}.\\
\end{align*}

By Corollary~\ref{Hess} the sign pattern
\[
\left[\begin{array}{rrrrr}
1 & -1 & 0 & 1 & 1 \\
1 & 1 & -1 & 0 & 1 \\
1 & 1 & 1 & -1 & 1 \\
1 & 1 & 1 & 1 & -1 \\
1 & 1 & 1 & 1 & 1 \\
\end{array}\right]
\]
allows orthogonality. However, there is no way to obtain an orthogonal matrix from an orthogonal Hessenberg with this sign pattern using a sequence of products of Givens rotations as described above. In particular, having a 1 in the $(1,5)$ and $(1,6)$ entries would require a nonzero sign in the $(1,3)$ or $(2,4)$ entry.

We next show that a row orthogonal matrix with the SIPP cannot have a large zero submatrix. First, we give a result about row orthogonal matrices.

\begin{lemma}\label{zeroLemma}
Let $Q\in\Orth(m,n)$ have the form 
\[
\renewcommand*{\arraystretch}{1.4}
Q =
\left[\begin{array}{c|c}
A & O\\
\hline
B & C
\end{array}\right], 
\]
where $O$ is $p\times q$. Then $p + q \leq n - {\rm rank}(B)$. In particular, if $B$ is nonzero, then $p+q \leq n-1$.
\end{lemma}
\begin{proof}
Since $Q$ is row orthogonal, the rows of $A$ are orthogonal to the rows of $B$. This, along with the observation that $\begin{bmatrix} A\\ \cmidrule(lr){1-1} B \end{bmatrix}$ is $m \times (n - q)$, gives
\[
\mbox{rank}(A) + \mbox{rank}(B) \leq n - q.
\]
Since the rows of $A$ are linearly independent, $\mbox{rank}(A) = p$. Thus,
\[
p + q \leq n - \mbox{rank}(B). \qedhere
\]
\end{proof}

Note that there exist $m\times n$ row orthogonal matrices with a $p\times q$ zero submatrix such that $p + q = n$. The next result shows such matrices do not have the SIPP.

\begin{proposition}\label{maxZeroBlock}
Let $Q\in\Orth(m,n)$ have a $p\times q$ zero submatrix. If $Q$ has the SIPP, then $p + q \leq n - 1$.
\end{proposition}
\begin{proof}
Suppose that $p + q \geq n$. By Lemma~\ref{zeroLemma}, $Q$ has the form
\[
\renewcommand*{\arraystretch}{1.4}
Q =
\left[\begin{array}{c|c}
A & O\\
\hline
O & B
\end{array}\right].
\]
Thus, by Lemma~\ref{basics}, $Q$ does not have the SIPP.
\end{proof}

The next two propositions show how to construct row orthogonal matrices with the SIPP and a $p\times q$ zero submatrix such that $p + q \leq n - 1$.

\begin{proposition} \label{OBlock}
Let $Q\in\Orth(m,n)$ have the form 
\[
\renewcommand*{\arraystretch}{1.4}
Q =
\left[\begin{array}{c|c}
A & O\\
\hline
B & C
\end{array}\right],
\]
where $B$ is nowhere zero and both $A$ and $C$ have the SIPP. Then $Q$ has the SIPP.
\end{proposition}

\begin{proof}
Let $X$ be an $m \times m$ symmetric matrix  such that $(XQ) \circ Q=O$. Partition $X$ as 
\[
\renewcommand*{\arraystretch}{1.4}
X= 
\left[ \begin{array}{c|c} X_1 & X_2\\ \hline
X_2^T & X_3 \end{array} 
\right],
\]
where $X_1$ (resp. $X_3$) is a symmetric $p\times p$ (resp. $(m-p) \times (m-p)$) matrix. As $(XQ) \circ Q=O$ and $B$ is nowhere zero, we have
\begin{eqnarray}
(X_1A+X_2B) \circ A&=&O \label{circ1}\\ 
X_2^TA+X_3B &=&O \label{circ2}\\ 
(X_3C) \circ C &=&O. \label{circ3}
\end{eqnarray}
The fact that $C$ has the SIPP, the symmetry of $X_3$ and (\ref{circ3}) imply that
\begin{equation} \label{x3c}
X_3 = O.
\end{equation}
Thus, (\ref{circ2}) becomes $X_2^TA = O$ and since the rows of $A$ are orthogonal, postmultiplying by $A^T$ gives 
\begin{equation} \label{x2}
X_2^T=O.
\end{equation}
Now (\ref{circ1}) simplifies to $(X_1A) \circ A=O$. The symmetry of $X_1$ and the fact that $A$ has the SIPP imply 
\begin{equation}\label{x1}
X_1 = O.
\end{equation}
Equations (\ref{x3c})-(\ref{x1}) imply that  $X=O$. Therefore, $Q$ has the SIPP.
\end{proof} 

We can use Proposition~\ref{removeRow} to obtain the next result.

\begin{proposition}
Let $M = \begin{bmatrix}A\\ \cmidrule(lr){1-1} \ba^T \end{bmatrix}$ and $N = \left[\begin{array}{@{}c|c@{}}\bb & B\end{array}\right]$ be row orthogonal matrices, where $\ba$ and $\bb$ are nowhere zero vectors. Then
\[
\renewcommand*{\arraystretch}{1.4}
Q =
\left[\begin{array}{c|c}
A & O\\
\hline
\bb\ba^T & B
\end{array}\right]
\]
is row orthogonal. Further, $Q$ has the SIPP if and only if $M$ and $N$ have the SIPP.
\end{proposition}
\begin{proof}
The matrix $Q$ is row orthogonal since
\[
\renewcommand*{\arraystretch}{1.4}
QQ^T = 
\left[\begin{array}{c|c}
AA^T & A\ba\bb^T\\
\hline
\bb\ba^TA^T & \bb\ba^T\ba\bb^T + BB^T
\end{array}\right]
=
\left[\begin{array}{c|c}
I & O\\
\hline
O & \bb\bb^T + BB^T
\end{array}\right]
= I.
\]
Assume that $M$ and $N$ have the SIPP. Suppose that
\[
\renewcommand*{\arraystretch}{1.4}
\left(
\left[\begin{array}{c|c}
X & Y\\
\hline
Y^T & Z
\end{array}\right]
\left[\begin{array}{c|c}
A &O\\
\hline
\bb\ba^T & B
\end{array}\right]
\right)
\circ
Q = O,
\]
where $X$ and $Z$ are symmetric. Then
\begin{align}
(XA + Y\bb\ba^T)\circ A &= O, \label{OB1}\\
Y^TA &= -Z\bb\ba^T, \text{ and} \label{OB2}\\
(ZB)\circ B &= O. \label{OB3}
\end{align}
Post-multiplying (\ref{OB2}) by $\ba$ gives $Z\bb = \bzero$. This, (\ref{OB2}) and the linear independence of the rows of $A$ imply that $Y = O$. By Proposition~\ref{removeRow}, $A$ has the SIPP. Hence~(\ref{OB1}) implies that $X = O$. Note $Z\left[\begin{array}{@{}c|c@{}}\bb & B\end{array}\right] = \left[\begin{array}{@{}c|c@{}}\bzero & BZ\end{array}\right]$. So $(Z\left[\begin{array}{@{}c|c@{}}\bb & B\end{array}\right])\circ \left[\begin{array}{@{}c|c@{}}\bb & B\end{array}\right]) = O$. Since $\left[\begin{array}{@{}c|c@{}}\bb & B\end{array}\right]$ has the SIPP, $Z = O$. Thus, we have shown $Q$ has the SIPP.

Now assume that $Q$ has the SIPP. Let $X_1$ and $X_2$ by symmetric matrices satisfying $(X_1 A)\circ A = O$ and $(X_2 N)\circ N = O$. Note $X_2\bb = \bzero$ and $X_2 B = O$. Thus,
\[
\renewcommand*{\arraystretch}{1.4}
\left(
\left[\begin{array}{c|c}
X_1 & O\\
\hline
O & X_2
\end{array}\right]
\left[\begin{array}{c|c}
A &O\\
\hline
\bb\ba^T & B
\end{array}\right]
\right)
\circ
Q = O.
\]
Since $Q$ has the SIPP, $X_1 = O$ and $X_2 = O$. Therefore, $N$ and $A$ have the SIPP. By Proposition~\ref{removeRow}, $M$ has the SIPP.
\end{proof}

For us, a \defi{hollow} matrix is a square matrix with zeros along the diagonal and nonzero entries off the diagonal. A \defi{signature matrix} is a diagonal matrix each of whose diagonal entries are either 1 or $-1$. Two matrices $A$ and $B$ are \defi{signature equivalent} if $A = D_1 B D_2$, where $D_1$ and $D_2$ are signature matrices.

\begin{theorem}\label{hollow}
Let $Q = [q_{ij}]\in\Orth(n)$ be hollow. Then $Q$ has the SIPP if and only if $Q$ is not signature equivalent to a symmetric hollow matrix.
\end{theorem}
\begin{proof}
We begin by proving that if $Q$ is signature equivalent to a symmetric hollow matrix then $Q$ does not have the SIPP. By Lemma~\ref{sign equiv} it suffices to assume $Q$ is a symmetric hollow matrix. Then $IQ^T$ is symmetric and $I\circ Q = O$. Thus, $Q$ does not have the SIPP.

Conversely, assume $Q$ does not have the SIPP. Then there exists a nonzero $X\in\R^{n\times n}$ such that $X\circ Q = O$ and $XQ^T$ is symmetric. Since $X$ is nonzero and $X\circ Q = O$ we know $X = \text{diag}(x_1,\ldots,x_n)$ with some $x_j\not=0$. Then some $x_k\not=0$ has the largest magnitude amongst $x_1,\ldots,x_n$. Since $XQ^T$ is symmetric $q_{ik} = \frac{x_i}{x_k}q_{ki}$ for $i=1,\ldots,n$.  Further, since $Q$ is orthogonal
\[
1 = \sum_{i=1}^n q_{ki}^2 \geq \sum_{i=1}^n\left(q_{ki}\frac{x_i}{x_k}\right)^2 = \sum_{i=1}^n q_{ik}^2 = 1.
\]
Thus, $x_1,\ldots,x_n$ must have the same magnitude, say $c\not=0$. It follows that $X = cD$ where $D$ is a signature matrix. Since $XQ^T$ is symmetric, $(QD)^T$ is symmetric. Thus, $Q$ is signature equivalent to a symmetric hollow matrix.
\end{proof}

Hollow orthogonal matrices can be used to demonstrate that a matrix having the SIPP depends on more than just its sign pattern.

\begin{example} \label{OMZD}
Consider the symmetric hollow, orthogonal matrix
\[
A =\frac1{\sqrt{5}}
\left[\begin{array}{rrrrrr}
0 & 1 & 1 & 1 & 1 & 1\\
1 & 0 & 1 & -1 & -1 & 1\\
1 & 1 & 0 & 1 & -1 & -1\\
1 & -1 & 1 & 0 & 1 & -1\\
1 & -1 & -1 & 1 & 0 & 1\\
1 & 1 & -1 & -1 & 1 & 0
\end{array}\right],
\]
and the hollow orthogonal matrix
\[
B = G_{12}(\pi/6) A G_{12}(\pi/6).
\]
By Theorem~\ref{hollow} $A$ does not have the SIPP. It is not difficult to verify that $B$ is a hollow orthogonal matrix with the same sign pattern as $A$. Further, one can see that $B$ is not signature equivalent to a symmetric matrix and therefore has the SIPP.
\end{example}

By Theorem~\ref{SIPPTran} we have the following corollary.

\begin{corollary}\label{SupHollow}
 Let $Q\in\Orth(n)$ have sign pattern $S$. If $Q$ is not signature equivalent to a symmetric hollow matrix, then every super pattern of $S$ allows orthogonality.
\end{corollary}

Recently hollow orthogonal matrices were used in \cite{BAILEY2018} to study the minimum number of distinct eigenvalues of certain families of symmetric matrices. In doing so they used the following construction.

\begin{lemma}[Bailey \textit{et al.}~\cite{BAILEY2018}]\label{CraigenConstruction}
Let
\[
\renewcommand*{\arraystretch}{1.4}
M = 
\left[\begin{array}{c|c}
A & \bv\\
\hline
\bu^T & 0
\end{array}\right] \text{ and }
N =
\left[\begin{array}{c|c}
0 & \bx^T\\
\hline
\by & B
\end{array}\right]
\]
be hollow orthogonal matrices of order $m+1$ and $n+1$ respectively. Then
\[
\renewcommand*{\arraystretch}{1.4}
Q =
\left[\begin{array}{c|c}
A & \bv\bx^T\\
\hline
\by\bu^T & B
\end{array}\right]
\]
is a hollow orthogonal matrix of order $m + n$.
\end{lemma}

Proposition~\ref{CraigenGen} is useful for verifying that matrices constructed using Lemma~\ref{CraigenConstruction} have the SIPP.

\begin{proposition}\label{CraigenGen}
Let
\[
\renewcommand*{\arraystretch}{1.4}
M = 
\left[\begin{array}{c|c}
A & \bv\\
\hline
\bu^T & 0
\end{array}\right] \text{ and }
N =
\left[\begin{array}{c|c}
0 & \bx^T\\
\hline
\by & B
\end{array}\right]
\]
be row orthogonal matrices with $\bx,\by,\bu$ and $\bv$ nowhere zero. Then
\[
\renewcommand*{\arraystretch}{1.4}
Q =
\left[\begin{array}{c|c}
A & \bv\bx^T\\
\hline
\by\bu^T & B
\end{array}\right]
\]
is row orthogonal. Moreover, if both $M$ and $N$ have the SIPP, then $Q$ has the SIPP.
\end{proposition}
\begin{proof}
The matrix $Q$ is row orthogonal since
\begingroup
\renewcommand*{\arraystretch}{1.4}
\begin{align*}
QQ^T &=
\left[\begin{array}{c|c}
AA^T + \bv\bx^T\bx\bv^T & A\bu\by^T + \bv\bx^TB^T\\
\hline
\by\bu^TA^T + B\bx\bv^T & \by\bu^T\bu\by^T + BB^T
\end{array}\right]\\
&=
\left[\begin{array}{c|c}
AA^T + \bv\bv^T & \bzero\by^T + \bv\bzero^T\\
\hline
\by\bzero^T + \bzero\bv^T & \by\by^T + BB^T
\end{array}\right]\\
&=I.
\end{align*}
\endgroup
Assume $M$ and $N$ have the SIPP. Let
\[
\renewcommand*{\arraystretch}{1.4}
X =
\left[\begin{array}{c|c}
X_1 & X_2\\
\hline
X_2^T & X_3
\end{array}\right]
\]
be a symmetric matrix such that $XQ\circ Q = O$. Since $XQ\circ Q = O$, and $\bx,\by,\bu$ and $\bv$ are nowhere zero
\begin{align}
(X_1 A + X_2\by\bu^T)\circ A &= O, \label{cg1}\\
X_1\bv\bx^T + X_2B &= O, \label{cg2}\\
X_2^T A + X_3\by\bu^T &=O, \label{cg3}\\
(X_2^T\bv\bx^T + X_3 B)\circ B &= O. \label{cg4}
\end{align}
Postmultiply~(\ref{cg2}) by $\bx$ to get 
\begin{equation}\label{cg5}
X_1\bv = \bzero \text{ and } X_2B = O.
\end{equation}
Similarly, postmultipying~(\ref{cg3}) by $\bu$  yields 
\begin{equation}\label{cg6}
X_3\by = \bzero \text{ and } X_2^T A = O.
\end{equation}
Equations~(\ref{cg1}), (\ref{cg5}) and (\ref{cg6}) imply
\begin{align*}
\renewcommand*{\arraystretch}{1.4}
\left(
\left[\begin{array}{c|c}
0 & \by^T X_2^T\\
\hline
X_2\by & X_1
\end{array}\right]
M
\right)
\circ M
&=
\renewcommand*{\arraystretch}{1.4}
\left[\begin{array}{c|c}
\by^T X_2^T \bv & \by^TX_2^TA\\
\hline
X_1\bv & X_2\by\bu^T + X_1 A
\end{array}\right]
\circ M\\
&=
\renewcommand*{\arraystretch}{1.4}
\left[\begin{array}{c|c}
 0 & \bzero^T\\
\hline
\bzero & (X_2\by\bu^T + X_1 A)\circ A
\end{array}\right]
\\
&= O.
\end{align*}
Since $M$ has the SIPP, $X_1 = O$ and $X_2\by = \bzero$. Similarly, equations~(\ref{cg4}), (\ref{cg5}) and (\ref{cg6}) imply $X_3 = O$. Having established $X_2\by = \bzero$, (\ref{cg5}) implies $X_2 \left[\begin{array}{@{}c|c@{}}\by & B\end{array}\right] = O$. Since the rows of $\left[\begin{array}{@{}c|c@{}}\by & B\end{array}\right]$ are linearly independent, $X_2 = O$. Therefore, $X = O$ and so $Q$ has the SIPP.
\end{proof}

The following lemma is a slight modification of Theorem 2.3 in \cite{BAILEY2018}.

\begin{lemma}\label{HollowSigEquiv}
There exists a hollow orthogonal matrix of order $n$ that is not signature equivalent to a symmetric matrix if and only if $n\notin\{1,2,3\}$.
\end{lemma}
\begin{proof}
The proof is by induction on $n$. It is easy to verify that hollow orthogonal matrices of order $n =1,3$ do not exist and that a non-symmetric hollow orthogonal matrix of order $n=2$ does not exist. Examples of hollow orthogonal matrices, not equivalent to a symmetric matrix, of order $n = 4,5$ are provided in \cite{BAILEY2018}. In particular, for $n=4$ we have
\[
A = 
\left[\begin{array}{rrrr}
0 & 1 & 1 & 1\\
1 & 0 & -1 & 1\\
1 & 1 & 0 & -1\\
1 & -1 & 1 & 0\\
\end{array}\right]
\]
and for $n = 5$ we have
\[
C =
\left[\begin{array}{ccccc}
0 & 1 & 1 & 1 & 1\\
1 & 0 & a & 1 & b\\
1 & b & 0 & a & 1\\
1 & 1 & b & 0 & a\\
1 & a & 1 & b & 0\\
\end{array}\right],
\]
where $a = \tfrac{-1 + \sqrt{3}}{2}$ and $b = \tfrac{-1 - \sqrt{3}}{2}$.

Suppose that there exist hollow orthogonal matrices, that are not signature equivalent to a symmetric matrix, for every order $m < k$ (except $m=1, 2, 3$), where $k\geq 6$. Let $B$ be such a matrix of order $k - 2$. By Theorem~\ref{hollow}, $A$ and $B$ have the SIPP. Allowing $A$ and $B$ to play the role of $M$ and $N$ in Lemma~\ref{CraigenConstruction} and Proposition~\ref{CraigenGen} produces a hollow orthogonal matrix of order $k$ that is not signature equivalent to a symmetric matrix.
\end{proof}

In order to understand the role of hollow orthogonal matrices in \cite{BAILEY2018} we require a few definitions. Let $A\in\sym n$. The \textit{graph of} $A$ is the simple graph with vertices $1,\ldots, n$ such that vertex $i$ is adjacent to vertex $j$ if and only if the $(i,j)$ entry of $A$ is nonzero. The set $\mathcal{S}(G)$ denotes the set of symmetric matrices with graph $G$. Let $q(G)$ denote the the minimum number of distinct eigenvalues of a matrix in $\mathcal{S}(G)$. Let $\alpha = \{1,\ldots, m\}$ and $\beta = \{1^\prime,\ldots,n^\prime\}$. For a bipartite graph $G$ with bipartition $\alpha \cup \beta$ let $\mathcal{B}(G)$ be the set of $m\times n$ matrices $B$, with rows and columns indexed by $\alpha$ and $\beta$ respectively, such that the $(i,j)$ entry of $B$ is nonzero if and only if $i$ and $j^\prime$ are adjacent in $G$.

\begin{theorem}[Ahmadi \textit{et al.}~\cite{Ahmadi2013}]\label{qG2}
Let $G$ be a bipartite graph with bipartition $\alpha\cup\beta$. Then $q(G) = 2$ if and only if $|\alpha| = |\beta|$ and there exists an orthogonal matrix $B\in\mathcal{B}(G)$.
\end{theorem}

Let $G_n$ denote the graph obtained from the complete bipartite graph $K_{n,n}$ by deleting a perfect matching. Using Lemma~\ref{CraigenConstruction} and Theorem~\ref{qG2} the authors of \cite{BAILEY2018} show the following.

\begin{theorem}[Bailey \textit{et al.}~\cite{BAILEY2018}]
Where $G_n$ is as above, $q(G_n) = 2$ unless $n=1$ or $n=3$.
\end{theorem}

Let $G_{n,k}$ denote the bipartite graph obtained by deleting a matching of size $k$ from $K_{n,n}$. In order to establish that $q(G_{n,k}) = 2$ the authors of \cite{BAILEY2018} constructed many orthogonal matrices. However, this observation is a simple consequence of the SIPP. A \defi{partially hollow matrix} to be any matrix with no zero entries off the main diagonal.

\begin{theorem}[Bailey \textit{et al.}~\cite{BAILEY2018}]
Suppose that $n\geq 1$ and $1\leq k \leq n$. Let $G_{n,k}$ be defined as above. Then $q(G_{n,k}) = 2$ if and only if $(n,k)\notin\{(1,1),(2,1),(3,2),(3,3)\}$.
\end{theorem}
\begin{proof}
By Theorem~\ref{qG2} $q(G_{n,k}) = 2$ if and only if there exists an orthogonal partially hollow matrix of order $n$ with exactly $k$ zero entries. For values of $n,k\leq 3$ the result is readily verifiable. The claim now follows from Lemma~\ref{HollowSigEquiv} and Corollary~\ref{SupHollow}.
\end{proof}

%
%
\section{Verification Matrix and Matrix Liberation}\label{SecVerif}

As we have seen, there are orthogonal matrices that do not have the SIPP. Recall that $Q\in\Orth(m,n)$, with sign pattern $S$, has the SIPP if and only if the manifolds $\Orth(m,n)$ and $\Q(S)$ intersect transversally at $Q$. In the setting where $Q$ does not have the SIPP we can replace $\Q(S)$ with an appropriately chosen manifold and apply the techniques of Section~\ref{develop} to obtain a result similar to Theorem~\ref{SIPPTran}. This new result, Theorem~\ref{AlgVerif}, allows us to determine some super patterns of $S$ that allow orthogonality.

The choice of manifold replacing $\Q(S)$ has some motivation. The best linear approximation to $\Orth(m,n)$ at $Q$ is $\T_{\Orth(m,n)\cdot Q}$. If we perturb $Q$ in the  direction of a matrix in $\T_{\Orth(m,n)\cdot Q}$ we can hopefully adjust the entries, all without changing the signs of the nonzero entries, so that we still have an orthogonal matrix. We will need to include a subspace of $\T_{\Orth(m,n)\cdot Q}$ in our new manifold. Many of the techniques in this section are motivated by the work in \cite{BARRETT2017}.

We begin by codifying $\T_{\Orth(m,n)\cdot Q}$ as the column space of an appropriately chosen matrix. Let $P\in\Orth(n)$ satisfy $QP^T = \left[\begin{array}{@{}c|c@{}}I & O\end{array}\right]$ and for $i\not=j$ define $K_{ij}$ to be the $m\times n$ matrix $E_{ij} - E_{ji}$. By Lemma~\ref{TOmn} the matrices
\[
B_{ij} =
\begin{cases}
K_{ij}P & \text{if } 1\leq i<j\leq m, \text{ and}\\
E_{ij}P & \text{if } m< j\leq n \text{ and } 1\leq i \leq m
\end{cases}
\]
form a basis for $\T_{\Orth(m,n)\cdot Q}$. For this choice of $P$, we define the \defi{tangent space matrix} $\ts_P(Q)$ \defi{of} $Q$ to be the $mn\times (mn - \binom{m+1}{2})$ matrix whose $(i,j)$-column is $\ve(B_{ij})$.

The matrix $\ts_P(Q)$ encodes more information than we need. For any set $E$ of pairs $(i,j)$ satisfying $1\leq i\leq m$ and $1\leq j \leq n$, $\ve_E(A)$ is the subvector of $\ve(A)$ of dimension $|E|$ that contains only the entries corresponding to the indices in $E$. Let $Z=\{(i,j): q_{ij} = 0\}$ and $p = |Z|$. The  \defi{tangent verification matrix} $\Psi_P(Q)$ \defi{of} $Q$ is the $p\times  (mn - \binom{m}{2})$ matrix $\ts_P(Q)[Z,:]$. That is, $\Psi_P(Q)$ is the restriction of $\ts_P(Q)$ to the rows corresponding to the zero entries of $Q$. If $Q\in\Orth(n)$ then $P$ is uniquely determined and we will write $\Psi(Q)$ in place of $\Psi_P(Q)$.

We can also represent $\norm_{\Orth(m,n)\cdot Q}$ as the column space of a matrix. By Lemma~\ref{NOmn} the $m\times n$ matrices
\[
C_{ij} = 
\begin{cases}
(E_{ij} + E_{ji})Q & \text{if } i < j\\
E_{ij}Q & \text{if } i=j,
\end{cases}
\]
form a basis for $\norm_{\Orth(m,n)\cdot Q}$. The \defi{normal space matrix} $\ns(Q)$ of $Q$ is the matrix whose $(i,j)$ column is $\ve(C_{ij})$. Let $E = \{(i,j): q_{ij} \not= 0\}$. The \defi{normal verification matrix} $\Omega(Q)$ of $Q$ is the matrix $\ns(Q)[E,:]$. 
\begin{observation}
When applying Lemma~\ref{MLL} to $\ns(Q)$ we require matrices $B\in \T_{\Orth(m,n)\cdot Q}$. Observe that $B\in \T_{\Orth(m,n)\cdot Q}$ if and only if $\ve(B)^T\ns(Q) = \bzero$.
\end{observation}

We require the following lemma from \cite{Holst2009}.

\begin{lemma}[Holst \textit{et al.}~\cite{Holst2009}]\label{trans2}
Assume $\M_1$ and $\M_2$ are manifolds that intersect transversally at $\by$ and let $\bv$ be a common tangent to each of $\M_1$ and $\M_2$ with $\|\bv\| = 1$. Then for every $\epsilon > 0$ there exists a point $\by^\prime \not= \by$ such that $\M_1$ and $\M_2$ intersect transversally at $\by^\prime$, and
\[
\left\|\frac{1}{\|\by-\by^\prime\|}(\by-\by^\prime)-\bv\right\| < \epsilon.
\]
\end{lemma}

The Matrix Liberation Lemma, below, is named after Lemma 7.3 in \cite{BARRETT2017}. This technical lemma is useful for working with specific matrices. It is also used to prove the more algebraic result Theorem~\ref{AlgVerif}.

\begin{lemma}[Matrix Liberation Lemma]\label{MLL}
Let $Q\in\Orth(m,n)$ have sign pattern $S$ and $B\in\T_{\Orth(m,n)\cdot Q}$ have sign pattern $R$. Let $P\in\Orth(n)$ satisfy $QP^T=\left[\begin{array}{@{}c|c@{}}I & O\end{array}\right]$ and $E = \{(i,j): b_{ij} \not= 0 \text{ or } q_{ij} \not= 0\}$, where $b_{ij}$ and $q_{ij}$ are the $(i,j)$ entries of $B$ and $Q$ respectively. Then every super pattern of $S_{\vec{R}}$ allows orthogonality provided
\begin{enumerate}
\item[\rm (i)]
the complement of the support of $\ve(S_{\vec{R}})$ corresponds to a linearly independent set of rows in $\Psi_P(Q)$; or equivalently

\item[\rm (ii)]
the columns of $NS(Q)[E,:]$ are linearly independent.
\end{enumerate}

\end{lemma}
\begin{proof}
Without loss of generality we may assume $\|B\| = 1$. Let $s_{ij}$ and $r_{ij}$ denote the $(i,j)$ entry of $S$ and $R$ respectively. Define the smooth manifold
\[
\M = \Q(S) + \spn{E_{ij} \in\R^{m\times n} : s_{ij} = 0 \text{ and } r_{ij} \not=0}.
\]
Let $t_{ij}$ denote the $(i,j)$ entry of $S_{\vec{R}}$. Then the tangent space 
\begin{equation}\label{TanM}
\T_{\M\cdot Q} = \spn{E_{ij} \in\R^{m\times n} : t_{ij}\not=0}
\end{equation}
and normal space
\begin{equation}\label{NormM}
\norm_{\M\cdot Q} = \spn{E_{ij} \in\R^{m\times n} : t_{ij}=0}.
\end{equation}

Assume (i) holds. Notice that the complement of the support of $\ve(S_{\vec{R}})$ is the set $Z = \{(i,j): t_{ij} = 0\}$. Let $N_0 = |Z|$. By assumption, the rows in $\Psi_P(Q)[Z,:]$ are linearly independent. Hence the columns of $\Psi_P(Q)[Z,:]$ span $\R^{N_0}$. By (\ref{TanM})
\[
\T_{M\cdot Q} + \T_{\Orth(m,n)\cdot Q}\supseteq\spn{E_{ij} \in\R^{m\times n} : t_{ij} = 0}
\]
and so $\T_{M\cdot Q} + \T_{\Orth(m,n)\cdot Q} = \R^{m\times n}$.

Now assume (ii) holds. Let $A \in \norm_{\M\cdot Q}\cap\norm_{\Orth(m,n)\cdot Q}$. By (\ref{NormM}) the $(i,j)$ entry of $A$ is zero whenever $t_{ij}\not = 0$. Notice that $E = \{(i,j): t_{ij}\not=0\}$. Having assumed the columns of $NS(Q)[E,:]$ are linearly independent, the only solution to
\[
\sum_{i \leq j} x_{ij} \ve_E(B_{ij}) = \bzero
\]
is the trivial solution (each $x_{ij} = 0$). Since $A\in \norm_{\Orth(m,n)\cdot Q}$ and the vectors $B_{ij}$ form a basis for $\norm_{\Orth(m,n)\cdot Q}$ it follows that $A = O$. Thus, $\norm_{\M\cdot Q}\cap\norm_{\Orth(m,n)\cdot Q} = \{O\}$.

In both cases $\M$ and $\Orth(m,n)$ intersect transversally at $Q$. Observe that $-B$ is a common tangent to $\Orth(m,n)$ and $\M$. Lemma ~\ref{trans2} guarantees that for every $\epsilon > 0$ there exists some $Y=[y_{ij}]$ such that $\Orth(m,n)$ and $\M$ intersect transversally at $Y$ and
\begin{equation}\label{ineq}
\left\|\frac{1}{\|Q-Y\|}(Q-Y)+B\right\| < \epsilon.
\end{equation}
Given that $Y\in\Orth(m,n)$ it remains to show that $\sgn(Y) = S_{\vec{R}}$. Since $Y \in \M$ it follows that $\sgn(y_{ij}) = s_{ij}$ whenever $s_{ij}\not=0$, and $y_{ij} = 0$ whenever $s_{ij}=r_{ij} = 0$. Suppose $s_{ij} = 0$ and $r_{ij}\not=0$. It follows from (\ref{ineq}) that $|b_{ij} - cy_{ij}|<\epsilon$, where $c = 1/\|Q - Y\|$. For $\epsilon$ small enough, $\sgn(y_{ij}) = r_{ij}$. Thus, $Y\in\Orth(m,n)$ has sign pattern $S_{\vec{R}}$. 

It remains to show that every super pattern of $S_{\vec{R}}$ allows orthogonality. Since the rows of $\Psi_P(Q)$ corresponding to the complement of the support of $\ve(S_{\vec{R}})$ are linearly independent, every super pattern $\hat{S}$ of $S_{\vec{R}}$ is the sign pattern of a matrix in $\T_{\Orth(m,n)\cdot Q}$. Further, the rows of $\Psi_P(Q)$ corresponding to the complement of the support of $\hat{S}$ are linearly independent. By the preceding argument, $\hat{S}$ allows orthogonality.
\end{proof}

The following observations are useful for working with verification matrices and the Matrix Liberation Lemma. Let $Q, P$ and $E$ be defined as in Lemma~\ref{MLL}.

\begin{observation}\label{reduceVerif}
When computing the verification matrix of $Q$ it is prudent to label the rows and columns of the verification matrix. The columns correspond to specific basis elements (from either the tangent or normal space) and the rows correspond to specific entries of $Q$. Permuting the columns and rows of a verification matrix preserves all relevant data. However, in doing so, it is necessary to record how the labels change.
\end{observation}

\begin{observation}\label{reduceVerif2}
Reducing the columns of $\Psi_P(Q)$ to a linearly independent set preserves the linear dependencies amongst the rows of $\Psi_P(Q)$. Similarly, reducing the rows of $\ns(Q)[E,:]$ to a linearly independent set preserves the linear dependencies amongst the columns of $\ns(Q)[E,:]$. 
\end{observation}

Theorem~\ref{TanVerif} is a restatement of Theorem~\ref{SIPPTran} in terms of the verification matrices. This formulation can be easily implemented with computer software to verify if a given matrix has the SIPP \cite{CURTIS2019}.

\begin{theorem}\label{TanVerif}
Let $Q\in\Orth(m,n)$ have sign pattern $S$ and $P\in\Orth(n)$ satisfy $QP^T=\left[\begin{array}{@{}c|c@{}}I & O\end{array}\right]$. 
\begin{enumerate}
\item[\rm (i)]
If the rows of $\Psi_P(Q)$ are linearly independent, then every super pattern of $S$ allows orthogonality.
\item[\rm (ii)]
If the columns of $\Omega(Q)$ are linearly independent, then every super pattern of $S$ allows orthogonality.
\end{enumerate}
\end{theorem}
\begin{proof}
Begin by assuming that the rows of $\Psi(Q)$ are linearly independent. Let $N_0$ denote the number of zero entries in $Q$. Then the columns of $\Psi(Q)$ span $\R^{N_0}$. Thus, for every super pattern $S^\prime$ of $S$ there exists some $B\in\T_{\Orth(m,n)\cdot Q}$ with sign pattern $R$ such that $S^\prime = S_{\vec{R}}$. Further, by Lemma~\ref{MLL} we know $S_{\vec{R}}$ allows orthogonality. Thus, (i) holds.

Now assume that the columns of $\Omega(Q)$ are linearly independent. Let $q_{ij}$ denote the $(i,j)$ entry of $Q$ and $A\in\norm_{\Q(S)\cdot Q}\cap \norm_{\Orth(m,n)\cdot Q}$. By Theorem~\ref{SIPPTran} it suffices to show that $A = O$. Observe that $\norm_{\Q(S)\cdot Q}$ consists of matrices whose $(i,j)$-entry is 0 whenever $q_{ij}\not=0$. Thus, the $(i,j)$-entry $A$ is 0 whenever $q_{ij}\not=0$. We must show the remaining entries of $A$ are also 0.

Suppose the columns of $\Omega(Q)$ are linearly independent. Let $E = \{(i,j):q_{ij} \not= 0\}$. Then the only solution to
\[
\sum_{i < j}x_{ij} \ve_E(B_{ij}) = 0
\]
is the trivial solution. Since the vectors $B_{ij}$ form a basis for $\norm_{\Orth(m,n)\cdot Q}$, $A = O$. Thus, (ii) holds.
\end{proof}

The following example illustrates how to use the verification matrices and shows that an $n\times n$ orthogonal matrix can have as few as eight zero entries and still not have the SIPP.

\begin{example}\label{barrier}
Let $\bu,\bv\in\R^n$ be nowhere zero. Suppose $\bu^T\bu=1$, $\bv^T\bv=1$, $\bu^T\bv=0$ and that $I-\bu\bu^T-\bv\bv^T$ is nowhere zero. Let $\omega = 1/\sqrt{2}$ and
\[
Q =
\left[\begin{array}{rrrr|c}
-1/2 & 1/2 & 0 & 0 & \omega \bu^T\\
1/2 & -1/2 & 0 & 0 & \omega \bu^T\\
0 & 0 & -1/2 & 1/2 & \omega \bv^T\\
0 & 0 & 1/2 & -1/2 & \omega \bv^T\\
\hline
\bigstrut \omega \bu & \omega \bu & \omega \bv & \omega \bv & I-\bu\bu^T-\bv\bv^T
\end{array}\right].
\]

It is routine to check that $Q\in\Orth(n+4)$. By Observations~\ref{reduceVerif} and~\ref{reduceVerif2} we may represent the verification matrix $\Psi(Q)$ as

\[
\begin{blockarray}{rrrrrrrrr}
 & B_{1,3} & B_{1,4} & B_{2,3} & B_{2,4} & B_{1,5} & B_{2,5} & B_{3,5} & B_{4,5} \\
\begin{block}{c[rrrrrrrr]}
(3,1) & 1 & 0 & -1 & 0 & 0 & 0 & 1 & 0 \topstrut\\
(4,1) & 0 & 1 & 0 & -1 & 0 & 0 & 0 & 1 \\
(3,2) & -1 & 0 & 1 & 0 & 0 & 0 & 1 & 0 \\
(4,2) & 0 & -1 & 0 & 1 & 0 & 0 & 0 & 1 \\
(1,3) & -1 & 1 & 0 & 0 & 1 & 0 & 0 & 0 \\
(2,3) & 0 & 0 & -1 & 1 & 0 & 1 & 0 & 0 \\
(1,4) & 1 & -1 & 0 & 0 & 1 & 0 & 0 & 0 \\
(2,4) & 0 & 0 & 1 & -1 & 0 & 1 & 0 & 0 \botstrut\\
\end{block}
\end{blockarray}.
\]

A subset of the rows of $\Psi(Q)$ are linearly dependent if and only if there exists a vector in the left nullspace of $\Psi(Q)$ whose support corresponds to the rows. The left nullspace of $\Psi(Q)$ is spanned by
\[
\bx^T =
\left[\begin{array}{rrrrrrrr}
1& -1& -1& 1& 1& -1& -1& 1
\end{array}\right].
\]
Thus, by Theorem~\ref{TanVerif}, $Q$ does not have the SIPP. 

Let $S$ be the sign pattern of $Q$. Since the rows of $\Psi(Q)$ form a minimal linearly dependent set, we may use any matrix $B\in\T_{\Orth(m,n)\cdot Q}$ when applying Lemma~\ref{MLL}. Observe that the column space of $\Psi(Q)$ is the orthogonal complement of the span of $\bx$. Further, the sign patterns of vectors which are orthogonal to $\bx$ are precisely those which are potentially orthogonal to $\bx$. Therefore, every super pattern of $S$ allows orthogonality except those of the form $S_{\vec{R}}$, where $R$ is nonzero and of the form
\[
{\renewcommand*{\arraystretch}{1.4}
R =
\left[\begin{array}{c|c}
\hat{R} & O\\
\hline
O & O
\end{array}\right]}
\text{ with }
\hat{R} = 
\left[\begin{array}{rrrr}
0 & 0 & \alpha_3 & \beta_4 \\
0 & 0 & \beta_3 & \alpha_4 \\
\alpha_1 & \beta_2 & 0 & 0 \\
\beta_1 & \alpha_2 & 0 & 0 
\end{array}\right],
\]
and each $\alpha_i=\pm 1, 0$ and each $\beta_i=\mp1, 0$.
\end{example}

We now apply the Matrix Liberation Lemma to obtain the following result.

\begin{theorem}\label{AlgVerif}
Let $Q\in\Orth(m,n)$ have sign pattern $S$ and $B\in\T_{\Orth(m,n)\cdot Q}$ have sign pattern $R$. If $X = O$ is the only $m\times m$ symmetric matrix satisfying $(XQ)\circ S_{\vec{R}} = O$, then every super pattern of $S_{\vec{R}}$ allows orthogonality.
\end{theorem}
\begin{proof}
Suppose that $X = O$ is the only $m\times m$ symmetric matrix satisfying $(XQ)\circ S_{\vec{R}} = O$. For $1\leq i \leq j \leq m$ define the $m\times m$ symmetric matrix
\[
A = \sum_{i\leq j} a_{ij}(E_{ij} + E_{ji}),
\]
where each $a_{ij}\in \R$. Let $E = \{(i,j): b_{ij} \not= 0 \text{ or } q_{ij} \not= 0\}$, where $b_{ij}$ and $q_{ij}$ are the $(i,j)$ entries of $B$ and $Q$ respectively. Then the columns of $NS(Q)[E,:]$ are linearly independent if and only if
\[
O = \left(\sum_{i\leq j}a_{ij}(E_{ij} + E_{ji})Q\right)\circ S_{\vec{R}} = (AQ)\circ S_{\vec{R}}
\]
implies $A = O$. This holds by assumption. Thus, by Lemma~\ref{MLL}, every super pattern of $S_{\vec{R}}$ allows orthogonality.
\end{proof}

Notice that the conditions in Theorem~\ref{AlgVerif} are very similar to the requirements of having the SIPP. Just as with the SIPP, we have a convenient result when dealing with square matrices.

\begin{corollary}\label{SquareVerif}
Let $Q\in\Orth(n)$ have sign pattern $S$ and $B\in\T_{\Orth(n)\cdot Q}$ have sign pattern $R$. If $Y = O$ is the only $m\times m$ matrix satisfying $YQ^T$ is symmetric and $Y\circ S_{\vec{R}} = O$, then every super pattern of $S_{\vec{R}}$ allows orthogonality.
\end{corollary}

An $(n+1)\times(n+1)$ orthogonal matrix can be obtained from $Q\in\Orth(n)$ with the direct sum $[1]\oplus Q$. While such matrices do not have the SIPP (see Lemma~\ref{basics}), they do have enough structure to produce the next result.

\begin{corollary}\label{addVertex}
Let $\bk\in\R^n$ and $Q\in\Orth(n)$ be nowhere zero. Define the sets
\[
D = \{\bx\in\R^n : (\bk^T Q)\circ \bx^T = \bzero\} \text{ and } F = \{Q^T\by \in\R^n: \bk\circ\by = \bzero\}
\]
If $D \cap F = \{\bzero\}$, then every super pattern of
\[
\renewcommand*{\arraystretch}{1.4}
\left[\begin{array}{c|c}
1 & \sgn(\bk^T Q)\\
\hline
-\sgn(\bk) & \sgn(Q)
\end{array}\right]
\]
allows orthogonality.
\end{corollary}
\begin{proof}
Assume $D\cap F = \{\bzero\}$. Define $\hat{Q} = [1] \oplus Q$ and $K\in\skw{n+1}$ by
\[
\renewcommand*{\arraystretch}{1.4}
K =
\left[\begin{array}{c|c}
0 & \bk^T\\
\hline
-\bk & O
\end{array}\right].
\]
Then 
\[
\renewcommand*{\arraystretch}{1.4}
K\hat{Q} =
\left[\begin{array}{c|c}
0 & \bk^T Q\\
\hline
-\bk & O
\end{array}\right]
\]
is in $\T_{\Orth(n+1)\cdot\hat{Q}}$. Let $S = \sgn(\hat{Q})$, $R = \sgn(K\hat{Q})$ and $X\in\R^{m\times m}$. Assume that $X\hat{Q}^T$ is symmetric and $X\circ S_{\vec{R}} = O$. Observe that
\[
\renewcommand*{\arraystretch}{1.4}
S_{\vec{R}} = 
\left[\begin{array}{c|c}
1 & \sgn(\bk^T Q)\\
\hline
-\sgn(\bk) & \sgn(Q)
\end{array}\right].
\]
Since $X\circ S_{\vec{R}} = O$ we know $X$ has the form
\[
\renewcommand*{\arraystretch}{1.4}
X = 
\left[\begin{array}{c|c}
0 & \bx^T\\
\hline
\by &O
\end{array}\right],
\]
where $\bx,\by\in\R^n$. Then
\[
\renewcommand*{\arraystretch}{1.4}
X\hat{Q}^T =
\left[\begin{array}{c|c}
0 & \bx^T Q^T\\
\hline
\by & O
\end{array}\right],
\]
and since $X\hat{Q}^T$ is symmetric, $Q\bx = \by$. Thus, $X\circ S_{\vec{R}} = O$ is equivalent to 
\[ (\bk^TQ)\circ\bx^T = \bzero \text{ and } \bk \circ (Q\bx) = \bzero.
\]
Notice that
\[
F = \{\bx \in\R^n: \bk\circ (Q\bx) = \bzero\}.
\]
Since $D\cap F = \{\bzero\}$, $\bx = \bzero$ and so $X = O$. By Corollary~\ref{SquareVerif}, every super pattern of $S_{\vec{R}}$ allows orthogonality.
\end{proof}

The next example illustrates how to apply Corollary~\ref{addVertex}.

\begin{example}
Consider the orthogonal matrix
\[
\hat{Q} = \frac13
\left[\begin{array}{rrrr}
3 & 0 & 0 & 0\\
0 & 1 & 2 & 2\\
0 & 2 & 1 & -2\\
0 & 2 & -2 & 1
\end{array}\right]
\]
and the vector
\[
\bk =
\left[\begin{array}{c}
0\\ 2\\ 1
\end{array}\right].
\]
Let $Q = \hat{Q}[\alpha,\alpha]$, where $\alpha = \{2,3,4\}$. Observe that
\[
\bk^TQ =
\frac13\left[\begin{array}{rrr} 2 & 0 & -1\end{array}\right].
\]
Then
\[
D = \{\bx : (\bk^T Q)\circ \bx^T = \bzero\} = \left\{\left[\begin{array}{c} 0\\ x_1\\ 0\end{array}\right]: x_1\in\R\right\}
\]
and
\[
F = \{Q^T\by : \bk\circ\by = \bzero\} = \left\{\left[\begin{array}{c} y_1\\ 2y_1\\ 2y_1\end{array}\right]: y_1\in\R\right\}.
\]
Thus, $D\cap F = \{\bzero\}$ and by Corollary~\ref{addVertex} the sign pattern
\[
\left[\begin{array}{rrrr}
1 & 1 & 0 & -1\\
0 & 1 & 1 & 1\\
-1 & 1 & 1 & -1\\
-1 & 1 & -1 & 1
\end{array}\right]
\]
allows orthogonality.
\end{example}

In \cite{WATERS1996} it was asked for which $n\times n$ sign patterns $S$ is the determinant function constant on $\Orth(n)\cap\Q(S) = \{Q\in\Orth(n): \sgn(Q) =S\}$? It was shown that for any sign pattern $S$ with order $n\leq 4$ the determinant is constant on $\Orth(n)\cap\Q(S)$. The next example shows that for each $n\geq 7$ there exist orthogonal matrices with the same sign pattern and oppositely signed determinants.

\begin{example}\label{biplane}
The $7\times 7$ sign pattern
\[
S=
\left[\begin{array}{rrrrrrr}
1 & 0 & 1 & 0 & -1 & 0 & 1 \\
-1 & 1 & 0 & 1 & 0 & 0 & 1 \\
0 & 1 & 1 & 0 & 0 & 1 & -1 \\
1 & 0 & -1 & 1 & 0 & 1 & 0 \\
0 & -1 & 1 & 1 & 1 & 0 & 0 \\
-1 & -1 & 0 & 0 & -1 & 1 & 0 \\
0 & 0 & 0 & -1 & 1 & 1 & 1 \\
\end{array}\right]
\]
allows orthogonality since $Q = \frac12S$ is orthogonal. Let
\[
X=
\left[\begin{array}{ccccccc}
0 & x_1 & 0 & x_2 & 0 & x_3 & 0 \\
0 & 0 & x_4 & 0 & x_5 & x_6 & 0 \\
x _7 & 0 & 0 & x_8 & x_9 & 0 & 0 \\
0 & x_{10} & 0 & 0 & x_{11} & 0 & x_{12} \\
x_{13} &0 & 0 & 0 & 0 & x_{14} & x_{15} \\
0 & 0 & x_{16} & x_{17} & 0 & 0 & x_{18} \\
x_{19} & x_{20} & x_{21} & 0 & 0 & 0 & 0 \\
\end{array}\right],
\]
where each $x_i\in\R$. Observe that $X\circ Q = O$. By assuming $XQ^T$ is symmetric we obtain a homogeneous system of 21 linear equations in $x_1,\ldots,x_{21}$ whose coefficient matrix can be shown to have nonzero determinant. Thus, each $x_k = 0$ and by Theorem~\ref{equivSIPP} the matrix $Q$ has the SIPP. This can be verified by applying Theorem~\ref{TanVerif} (see \cite{CURTIS2019}).

Observe that $I - K$ is a super pattern of $S$, where $K$ is the skew symmetric matrix
\[
K = 
\left[\begin{array}{rrrrrrr}
0 & -1 & -1 & 1 & 1 & -1 & -1 \\
1 & 0 & 1 & -1 & -1 & -1 & -1 \\
1 & -1 & 0 & -1 & 1 & -1 & 1 \\
-1 & 1 & 1 & 0 & 1 & -1 & -1 \\
-1 & 1 & -1 & -1 & 0 & -1 & -1 \\
1 & 1 & 1 & 1 & 1 & 0 & -1 \\
1 & 1 & -1 & 1 & 1 & 1 & 0 \\
\end{array}\right].
\]

Since the determinant function is continuous and orthogonal matrices have determinant $\pm1$, any orthogonal matrix near $Q$ will have the same determinant as $Q$. In particular, any matrix with sign pattern $I-K$ obtained by applying Theorem~\ref{TanVerif} to $Q$ will have determinant $\det(Q) = -1$.

Using the Cayley transform, every matrix $P = (I - \epsilon K)(I + \epsilon K)^{-1}$, where $\epsilon > 0$, is orthogonal. Notice that for $\epsilon < 1$ 
\[
(I + \epsilon K)^{-1} = I + \sum_{i=1}^\infty (- \epsilon K)^i
\]
and $I - \epsilon K$ is nowhere zero. Thus, if we choose $\epsilon$ small enough $P$ will have sign pattern $I - K$. Since $P$ was obtained using the Cayley transform $\det(P) = 1$. Thus, $P$ and $Q$ both have sign pattern $I - K$ and $\det(P)\not=\det(Q)$.

Further, by using Corollary~\ref{addVertex} and induction we can guarantee the existence of irreducible orthogonal matrices of order $n\geq 7$ with the same sign pattern and oppositely signed determinant.
\end{example}

Let $Q = \bigoplus_{i=1}^k Q_i$, where each $Q_i\in\Orth(n_i)$. In \cite{WATERS1996} the implicit function theorem was used to find super patterns of $\sgn(Q)$ that allow orthogonality. The main result of \cite{WATERS1996}, Theorem 3.14, can be phrased in terms of the SIPP. In fact, using Theorem~\ref{AlgVerif} we obtain the following stronger result. Unlike Theorem 3.14 in \cite{WATERS1996}, Corollary~\ref{WatersResult} can produce sign patterns of orthogonal matrices by perturbing the zero entries in the blocks $Q_i$ of $Q$.

\begin{corollary}\label{WatersResult}
Let $Q = \bigoplus_{i=1}^k Q_i$ have sign pattern $S$ and let each $Q_i\in\Orth(n_i)$ where $\sum_{i=1}^k n_i = n$. Let $B\in\R^{n\times n}$ be a block matrix with $(i,j)$-block denoted as $B_{ij}\in\R^{n_i\times n_j}$ and $R = \sgn(B)$. If each $B_{ij} = -Q_i B_{ji}^TQ_j$, and $X=O$ is the only symmetric matrix satisfying $(XQ)\circ S_{\vec{R}} = O$, then $S_{\vec{R}}$ allows orthogonality.
\end{corollary}
\begin{proof}
Suppose that each $B_{ij} = -Q_i B_{ji}^TQ_j$. Then $B_{ij}Q_{j}^T = -Q_{i}B_{ji}^T$. Define $K_{ij} = B_{ij}Q_{j}^T$ so that $K_{ij} = -K_{ji}^T$. It follows that $K = [K_{ij}]$ is skew symmetric and that $B = KQ$. Thus, $B\in \T_{\Orth(n)\cdot Q}$. The claim now follows from Theorem~\ref{AlgVerif}.
\end{proof}

\bibliographystyle{elsarticle-num}
\bibliography{refs}

\end{document}